\documentclass[11pt]{article}
\usepackage{latexsym,color,xcolor,amsmath,amsthm,amssymb,amscd,amsfonts,stmaryrd,dsfont}
\usepackage{graphicx,bbm}
\usepackage{cancel, mathtools}

\newtheorem{conjecture}{Conjecture}[section]
\newtheorem{theorem}[conjecture]{Theorem}

\newtheorem{remark}[conjecture]{Remark}
    \newtheorem{lemma}[conjecture]{Lemma}
\newtheorem{proposition}[conjecture]{Proposition}

\newtheorem{definition}[conjecture]{Definition}

\newcommand\independent{\protect\mathpalette{\protect\independent}{\perp}} 
\def\independent#1#2{\mathrel{\rlap{$#1#2$}\mkern2mu{#1#2}}}

\date{\vspace{-5ex}}

\date{\today}

\begin{document}

\title{Rearrangements and infimum convolutions}
\author{Devraj Duggal, James Melbourne, and Cyril Roberto}
\maketitle
\begin{abstract}
We develop a general comparison result for inf-convolution operators related to rearrangement. As a consequence we derive comparison results under spherically symmetric rearrangement for Laplace and polar transforms. As a by product we simplify existing proofs related to the functional Blaschke-Santalo's inequality of Keith Ball and derive a comparison result for some parabolic PDE.
\end{abstract}

\section{Introduction}

In his seminal papers \cite{talenti,talenti76}, Talenti used spherically symmetric rearrangements to compute the optimal constant in  Sobolev inequalities (see also \cite{aubin}) and to obtain a comparison result between the solutions of the elliptic equations $-\Delta u=f$ and $-\Delta v=f^*$ (the rearranged function $f^*$ of $f$). 
These two results open wide lines of investigation in functional analysis and PDE theory. 
As is well-known, rearrangements proved rather fruitful in various field of mathematics that include functional analysis, PDE and Probability Theory. To give just a partial list of applications is out of reach and we refer the reader to the books \cite{lieb-loss,baerstein,leoni} for general accounts on rearrangement.
One key feature of rearrangements, illustrated by Talenti's papers, is that it allows one to reduce a $n$-dimensional problem to a one-dimensional one which becomes simpler and sometimes solvable.

\medskip

In this paper our aim is to proceed along the same lines. We will provide, after \cite{melbourne2019rearrangement}, some new general rearrangement comparison result related to inf-convolution operators. This will be applied to the comparison of Laplace and polar transforms and some comparison of solutions of a parabolic equation that are given by inf-convolutions. Also, after rearrangement and reduction to dimension one,  we may provide simple proofs of existing results in convex geometry, namely on the existence of maximizers in the functional form of the Blaschke-Santalo inequality and its polar counterpart, together with the best constant in the former inequality.

\medskip

In the next section (Section \ref{sec:notations-definitions}), we introduce the necessary material and in particular the notion of inf-convolution of interest for us, together with a general notion of rearrangement (that includes the classical spherically symmetric rearrangement on the Euclidean space, the Gaussian rearrangement of Ehrhard and rearrangement on the sphere, and many others).

\smallskip

Section \ref{sec:general} is dedicated to the statement and the proof of one of our main theorems (Theorem \ref{th:isoperimetric-rearrangeement}).

\smallskip

In Section  \ref{sec:transforms} we apply our main theorem to compare the action of Laplace and polar transforms on a function and its rearranged counterpart.

\smallskip

Finally Section \ref{sec:applications} collects some applications. The first one is related to solution of some parabolic equations. The second and third give a simple proof of a known result by Ben Li \cite{li} on the existence of extremizers in the functional form of Blaschke-Santalo's inequality \cite{ball86,artstein-klartag-milman,fradelizi-meyer,lehec} in convex geometry, and a proof of such an inequality through a new semi-group argument developed by Cordero et al. \cite{cordero,CFL}, after Sakamura-Tsuji \cite{nakamura}.

\section{Notations and definitions}
\label{sec:notations-definitions}

In this section we collect the necessary definitions that we will use throughout the paper. 

We assume that $(\mathcal{X},\nu)$ and $(\mathcal{Y}, \mu)$ are Polish measure spaces endowed with their Borel $\sigma$-algebra. A \emph{cost function} is a measurable function $\phi \colon \mathcal{X} \times \mathcal{Y} \to \mathbb{R}$. Observe that we relax here the usual definition of a cost function, which, in the transport theory, is usually assumed to take only non-negative values.

One of the central objects in this paper is the infimum convolution operator.

\begin{definition}[Infimum convolution] \label{def:inf-convolution}
    For $\phi$ a cost function on $\mathcal{X} \times \mathcal{Y}$, and $f: \mathcal{Y} \to \mathbb{R}$ define the \emph{infimum convolution} $Q_\phi f: \mathcal{X} \to \mathbb{R} \cup \{-\infty\}$ by
    \[
        Q_\phi f (x) \coloneqq \inf_{y \in \mathcal{Y}} \left\{ f(y) + \phi(x,y) \right\}.
    \]
\end{definition}

In many interesting situations related to transport theory or Hamilton Jacobi equations $\mathcal{X}=\mathcal{Y}=\mathbb{R}^n$ (or some smooth manifold) and the cost function is a function of the distance $\phi(x,y)=d(x,y)^p$ for $p \geq 1$, see \textit{e.g.}\ \cite{villani}. Then the infimum convolution $Q_\phi$ amounts to the classical Hopf-Lax formula \cite{evans}.
More general cost functions were considered in the literature, related to optimal transport in a weak sense \cite{GRST} with applications in economy \cite{Backhoff-Veraguas-Pammer}. In such applications, one considers $X=\mathbb{R}^n$ and Y to be a subset of all probability measures on $\mathbb{R}^n$.
Although such general costs enter the framework of our result, we will not consider them since no associated rearrangement seems known in the corresponding literature. 

In the next definition we introduce the notion of "enlargement" associated to a cost function.

\begin{definition}[Enlargement]
    For a cost function $\phi$ on $\mathcal{X} \times \mathcal{Y}$ and an $\varepsilon \in \mathbb{R}$, define for $A \subseteq \mathcal{Y}$,
    \[
        A_\varepsilon \coloneqq A_{\phi,\varepsilon} \coloneqq \{x \in \mathcal{X} : \exists y \in A \text{ s.t. }  \phi(x,y) < \varepsilon \} 
    \]
    and 
     \[
        A_{\overline{{\varepsilon}}} \coloneqq  A_{\overline{\phi,\varepsilon}} \coloneqq \{x \in \mathcal{X} : \exists y \in A \text{ s.t. }  \phi(x,y) \leq \varepsilon \}.
    \]
\end{definition}

When $\mathcal{X} = \mathcal{Y}$ and $\phi$ is a distance, then $A_\varepsilon$ and $A_{\overline{\varepsilon}}$ agree with usual notion of an $\varepsilon$-enlargement of a set in the literature. Moreover, in this case, $A_\varepsilon$ and $A_{\overline{\varepsilon}}$ are only non empty for $\varepsilon > 0$ and $\varepsilon \geq 0$ respectively.
Observe that in the definition of the enlargement of a set, we allow $\varepsilon$ to take negative values. This comes from the fact that, in some situations, the cost might take negative values.

We denote the open sets of a topological space $E$ by $\mathcal{O}(E)$ and the Borel sets (the members of the smallest $\sigma$-algebra containing $\mathcal{O}(E)$) of a topological space $E$, by $\mathcal{B}(E)$.

The second central objet of this paper is the notion of \emph{rearrangement}.

\begin{definition}[Rearrangement-Decreasing rearrangement]
    A map $*: \mathcal{B}(\mathcal{Y}) \to \mathcal{O}(\mathcal{Y})$ is a \emph{rearrangement} under a measure $\mu$ if $\mu(A) = \mu(*(A))$ and $\mu(A) \leq \mu(B)$ implies $*(A) \subseteq *(B)$.  Further for a measurable function $f: \mathcal{Y} \to [0,\infty)$, its $*$-symmetric \emph{decreasing rearrangement} is defined by
    \begin{align*}
        f^*(y) \coloneqq  \int_0^\infty \mathbbm{1}_{ *(\{ f> \lambda \})} (y) d \lambda, \qquad y \in \mathcal{Y}.
    \end{align*}
\end{definition}

For brevity of notation we write $A^* \coloneqq *(A)$.  We will further assume that the map $*$ satisfies 
\begin{align} \label{eq: decreasing rearrangment characterization}
    \{ f > \lambda \}^*  = \{ f^* > \lambda \}.
\end{align}  Note that $\{ f > \lambda \}^*  = \{ f^* > \lambda \}$ is ensured (see Proposition 2.7 of \cite{melbourne2019rearrangement}) if $*$ is a rearrangement such that $A_i \subseteq A_{i+1}$ implies 
\begin{align} \label{eq: rearrangement requirement}
        \left( \bigcup_{i=1}^\infty A_i \right)^* \ \subseteq \ \bigcup_{i=1}^\infty A_i^*.
\end{align}

In this paper we will mainly consider the spherically symmetric decreasing rearrangement on $\mathbb{R}^n$, spherical caps on the sphere, and the Ehrhard Gaussian rearrangement \cite{ehrhard}. 
For a more complete zoology of rearrangements, we refer the reader  to Section 2 of the book by Kawohl \cite{kawohl}.
For $A$ a Borel set of $\mathbb{R}^n$,
its  spherically symmetric rearrangement is the ball $A^*$ centered at $0$ with the same Lebesgue measure, while, the Gaussian rearrangement of $A$ is the half space
$A^*=\{x \in \mathbb{R}^n: x_1 \leq a\}$ with $a$ chosen so that 
$A^*$ has the same (standard) Gaussian measure as $A$. 

One can also define increasing rearrangement as follows. 

\begin{definition}[Increasing rearrangement]
    For $* \colon \mathcal{B}(\mathcal{Y}) \to \mathcal{O}(\mathcal{Y})$ a rearrangement with respect to a measure $\mu$, and a $\mu$-measurable function $f: \mathcal{Y} \to \mathbb{R}$, we define the $*$-symmetric \emph{increasing rearrangement} by,
    \[
        f_*(y) = - \log (e^{-f})^*.
    \]
\end{definition}

Note that by \eqref{eq: decreasing rearrangment characterization}, 
\begin{align} \label{eq: increasing rearrangement characterization}
    \{ f_* < \lambda \} = \{ f < \lambda \}^*
\end{align}
as $\{ f_* < \lambda \} = \{ (e^{-f})^* > e^{-\lambda} \} = \{ e^{-f} > e^{- \lambda } \}^* = \{ f < \lambda \}^*$.

As is well known, rearrangements are related to isoperimetric inequalities.
For the 3 classical isoperimetric inequalities alluded to above (Euclidean on $\mathbb{R}^n$, the sphere and the Gaussian space), the isoperimetric sets (realizing equality in the isoperimetric inequality) are nested 
in the sense that one contains the other and ,moreover, the set of extremal sets is totally ordered. This property leads to the following definition.




\begin{definition}
    \label{def:isoperimetric-rearrangement}
    We consider a map  $*: \mathcal{B}(\mathcal{Y}) \to \mathcal{O}(\mathcal{Y})$ to be an \emph{isoperimetric rearrangement} for the cost function $\phi$ and measure $\mu$, if $* (\mathcal{B}(\mathcal{Y}))$ is totally ordered and any Borel set $A \subseteq \mathcal{Y}$ and $\varepsilon \in \mathbb{R}$ satisfy
    \begin{align} \label{eq: rearrangement one way}
        \mu(A_{\phi,\varepsilon}) \geq \mu((A^*)_{\phi,\varepsilon}).
    \end{align}
    and
    \begin{align} \label{eq: rearrangement two way}
        \mu((A_{\phi,\varepsilon})^c) \leq \mu(((A^\star)_{\phi,\varepsilon})^c) 
    \end{align}
    In this case, we may say that $(*,\mu,\phi)$ is an isoperimetric rearrangement.
\end{definition}

In \eqref{eq: rearrangement one way} and \eqref{eq: rearrangement two way} we adopt the convention that 
$\infty \leq \infty$. 
In particular, if $\mu$ is a finite measure, then \eqref{eq: rearrangement one way} and \eqref{eq: rearrangement two way} are equivalent.  If $\mu$ is  not a finite measure, $\mu((A^*)_{\phi,\varepsilon})$ is always finite and \eqref{eq: rearrangement one way} holds, then \eqref{eq: rearrangement two way} holds trivially.  Similarly if $\mu(((A^\star)_{\phi,\varepsilon})^c) $ is always finite and \eqref{eq: rearrangement two way} holds, then $\mu((A_{\phi,\varepsilon})^c)$ is always finite and hence $\mu((A_{\phi,\varepsilon}))$ is always infinite implying \eqref{eq: rearrangement one way} trivially.

\begin{remark} \label{rem:transfer}
If $\Phi \colon \mathcal{X} \times \mathcal{Y} \to \mathbb{R}$ is a cost function and $\alpha \colon \mathrm{Im(\phi)} \to \mathbb{R}$ is a measurable one to one map from the image $\mathrm{Im(\phi)} \subset \mathbb{R}$ of $\phi$ onto its image $\mathrm{Im}(\alpha)$, then $\psi = \alpha(\Phi)$  is a cost function on $\mathcal{X} \times \mathcal{Y}$ and $(*,\mu,\phi)$
is an isoperimetric rearrangement/co-rearrangement if and only if 
$(*,\mu,\psi)$
is an isoperimetric rearrangement/co-rearrangement 
\end{remark}

We now illustrate  these definitions 
with 3 classical rearrangements, namely the spherically symmetric rearrangement on the Euclidean space, the Gaussian rearrangement of Ehrhard, and the symmetric rearrangement on the sphere.

\subsubsection*{Spherically symmetric rearrangements on the Euclidean space}
Consider $\mathcal{X}=\mathcal{Y}=\mathbb{R}^n$ equipped with the Lebesgue measure $\lambda$, \textit{i.e.}\ $\mu=\nu=\lambda$. For sets, we interchangeably use the notation $|\cdot|$ for the Lebesgue measure.

Set $B \coloneqq \{ x : \|x\| < 1 \}$ for the open unit ball of $\mathbb{R}^n$. Here and in the sequel $\|\cdot\|$ stands for the Euclidean norm.

As mentioned earlier, given a Borel set $A$ of $\mathbb{R}^n$, its spherically symmetric rearrangement is $A^*=rB=\{x \in \mathbb{R}^n : \|x\| < r\}$, with $r$ chosen so that $|A^*|=|A|$. 

The classical isoperimetric inequality for the Lebesgue measure on $\mathbb{R}^n$ can be seen, as is well known \cite{ledoux}, as a special case of the celebrated Brunn-Minkowski inequality. This inequality may be formulated as follows \cite{gardner}: for Borel sets $A,C$ of finite measure, their Minkowsky sum $A+C \coloneqq \{a+c, \; a \in , c \in C\}$ decreases on spherically decreasing rearrangement, \textit{i.e.}\ 
$$
|A+C| \geq |A^*+C^*|.
$$
Now, consider the cost $\phi(x,y)=\|x-y\|$.
Since spherically symmetric rearrangements are centered balls (that are totally ordered), applying the Brunn-Minkoswki inequality
to $A$ of finite Lebesgue measure and the unit ball $C=B$ and noting that 
$A_\varepsilon = A + \varepsilon B$
and
$(A^*)_\varepsilon = A^* + \varepsilon B$ (that is a ball), 
$$
|A_\varepsilon| \geq |(A^*)_\varepsilon |.
$$
In particular, the spherically symmetric rearrangement is an isoperimetric rearrangement for the cost $\|x-y\|$ and the Lebesgue measure $\lambda$ as in the sense of Definition \ref{def:isoperimetric-rearrangement}.
By Remark \ref{rem:transfer} this property transfers to any cost $\alpha(\|x-y\|)$ for a proper choice of the map $\alpha$.

\subsubsection*{Ehrhard Gaussian rearrangement on the Euclidean space}
Consider again $\mathcal{X}=\mathcal{Y}=\mathbb{R}^n$ with the cost $\phi(x,y)=\|x-y\|$ but now equipped with the standard Gaussian measure  $\gamma$ with density $\varphi(x)=e^{-|x|^2/2}/(2\pi)^{n/2}$, $x \in \mathbb{R}^n$. Let $\Phi(t)=\int_{-\infty}^t e^{-s^2/2}\frac{ds}{\sqrt{2\pi}}$, $t \in \mathbb{R}$. 

Given a Borel set $A$ of $\mathbb{R}^n$, its Gaussian rearrangement is 
$A^*=\{x \in \mathbb{R}^n : x_1 < a\}$
where $a=\Phi^{-1}(\gamma(A))$ is chosen so that $\gamma(A^*)=\gamma(A)$.
Observe that the choice of the first coordinate is arbitrary and one could equivalently consider rearrangements of the form 
$A^*=\{x \in \mathbb{R}^n : \langle x , u \rangle < a\}$, for any unit vector $u$, where $\langle\cdot,\cdot \rangle$ stands for the Euclidean scalar product.

Then, the Gaussian isoperimetric inequality of Borell \cite{borell}, Sudakov and Tsireslon \cite{sudakov}
asserts that for any Borel set $A$
$$
\gamma(A_\varepsilon) \geq \Phi( \Phi^{-1}(\gamma(A)) + \varepsilon).
$$
By construction 
$\Phi( \Phi^{-1}(\gamma(A)) + \varepsilon)) = \gamma(\{x : x_1 \leq \Phi^{-1}(\gamma(A)) + \varepsilon\})$.
On the other hand, 
for $y \in (A^*)_\varepsilon$ and 
$x \in A^*=\{x: x_1 \leq \Phi^{-1}(\gamma(A))\}$ with $\|x-y\| \leq \varepsilon$, it holds that
$y_1=x_1+y_1-x_1 \leq \Phi^{-1}(\gamma(A)) + \varepsilon$. Hence 
$(A^*)_\varepsilon \subset \{x : x_1 \leq \Phi^{-1}(\gamma(A)) + \varepsilon\}$ and in turn, from the Gaussian isoperimetric inequality
$$
\gamma(A_\varepsilon) \geq \gamma((A^*)_\varepsilon).
$$
In conclusion, the Ehrhard Gaussian rearrangement is an isoperimetric rearrangement for the cost $\|x-y\|$ and the standard Gaussian measure $\gamma$. 
Again from Remark \ref{rem:transfer}, this property transfers to more general costs $\alpha(\|x-y\|)$.

\subsubsection*{Isoperimetry on the sphere}
An other classical example is given by the sphere $\mathcal{X}=\mathcal{Y}=S^{n-1}\subseteq\mathbb{R}^{n}$ equipped with the normalized uniform measure $\mu=\nu=\sigma_{n-1}$. Let $d$ denote the standard geodesic distance on the sphere. For a Borel set $A\in\mathcal{B}(S^{n-1})$, let $A_{\varepsilon}=\{x\in S^{n-1}:d(x,A)<\varepsilon\}$ be the usual enlargement related to $d$. 
By convention, we speak of spherical caps centered at the north pole (\textit{i.e.}\ the geodesic balls centered at the north pole). The rearrangement of interest associates $A\subseteq S^{n-1}$ to $A^*\subseteq S^{n-1}$ where $A^*$ is the cap centered at the north pole such that $\sigma_{n-1}(A)=\sigma_{n-1}(A^*)$.
The classical Isoperimetry result due to Lévy (see \textit{e.g.}\ \cite{ledoux}) states that 
$$
\sigma_{n-1}(A_\epsilon)\geq\sigma_{n-1}((A^*)_{\epsilon}) .
$$ 
Therefore, $(*,\sigma_{n-1},d)$ is an isoperimetric rearrangement and so is  
$(*,\sigma_{n-1},\alpha(d))$ for proper choices of the map $\alpha$ by Remark \ref{rem:transfer}.

\section{A general Argument} \label{sec:general}

Our aim is to prove the following general theorem on isoperimetric rearrangements and inf-convolution which follows and extends the chief ideas of \cite{melbourne2019rearrangement}.
At the end of the section we will give some examples of applications together with a weaker form of our general result when only a non optimal isoperimetric inequality is available.

\begin{theorem} \label{th:isoperimetric-rearrangeement}
    For $(*, \mu, \phi)$ an isoperimetric rearrangement and $\lambda \in \mathbb{R}$,
    \[
        \mu (\{ Q_\phi f < \lambda \}) \geq \mu (\{ Q_\phi f_* < \lambda \}) 
   \qquad \mbox{and} \qquad 
        \mu (\{ Q_\phi f \geq \lambda \}) 
        \leq \mu (\{ Q_\phi f_* \geq \lambda \}) .
    \]
\end{theorem}

For an infinite measure $\mu$, the conclusion of the above statement may read as $\infty \geq \infty$ (or $\infty \leq \infty$) which holds true by virtue of our convention. 

Observe that for the symmetric decreasing rearrangement and the Lebesgue measure on $\mathbb{R}^n$, that are isoperimetric rearrangement for some costs, the conclusion of the theorem involves the symmetric increasing rearranged function $f_*$ and not $f^*$.

The proof of Theorem \ref{th:isoperimetric-rearrangeement}
relies on the following decomposition lemma. The main idea here is to obtain a countable union of sets. Set $\mathbb{Q}$ for the set of rational numbers.

\begin{lemma} \label{lem: decomposition}
    For $\lambda \in \mathbb{R}$, the sublevel sets of an infimum convolution have the following decomposition,
    \[
        \{ Q_\phi f < \lambda \} = \bigcup_{q \in S(\lambda)} \{ f < q_1 \}_{\phi,q_2}
    \]
    where
    \[
        S(\lambda) \coloneqq \left\{ q=(q_1,q_2) \in \mathbb{Q}^2 : q_1 + q_2 < \lambda \right\}.
    \]
\end{lemma}

\begin{proof}
    If $x \in \{ f < q_1 \}_{\phi,q_2}$, for some $q \in S(\lambda)$ then there exists $y \in \{ f < q_1\}$ such that $\phi(x,y)< q_2$, thus
    \[
        f(y) + \phi(x,y) < q_1 + q_2 < \lambda
    \]
    so that $Q_\phi f(x) < \lambda$, giving $\bigcup_{q \in S(\lambda)} \{ f < q_1 \}_{\phi,q_2} \subseteq \{ Q_\phi f < \lambda \}$.
    For $x$ such that $Q_\phi f(x) < \lambda$, by definition there exists $y' \in \mathcal{Y}$ such that $f(y') + \phi(x,y') < \lambda$.  For $\delta > 0$, such a $y'$ belongs to $\{ f < f(y') + \delta \}$, as well as $\{ y : \phi(x,y) < \phi(x,y') + \delta \}$, thus
    \[
        x \in \{ f < f(y') + \delta \}_{\phi,\phi(x,y') + \delta} 
    \] 
    For $\delta$ chosen small enough, $f(y') + \phi(x,y') + 2 \delta < \lambda$.  Choosing $q = (q_1,q_2)$ such that $q_1 \in \mathbb{Q} \cap (f(y'), f(y') + \delta)$ and $q_2 \in (\phi(x,y'), \phi(x,y') + \delta)$ gives $x \in \{ f < q_1 \}_{\phi,q_2}$ for this $q \in S(\lambda)$, completing the proof.
\end{proof}

\begin{proof}[Proof of Theorem \ref{thm: Blaschke-Santalo rearrangement bis}]
The proof is obtained from the following computation: by Lemma \ref{lem: decomposition} and hypothesis (by definition of the isoperimetric rearrangement)        
    \begin{align*}
        \mu (\{ Q_\phi f < \lambda \} )
            &=
                \mu \left( \bigcup_{q \in S(\lambda)} \{ f < q_1 \}_{\phi,q_2} \right)
                    \\
            &\geq 
                \sup_{q \in S(\lambda)} \mu  \left( \{ f < q_1 \}_{\phi,q_2} \right)
                   \\
            &\geq 
                \sup_{q \in S(\lambda)} \mu  \left( (\{ f < q_1 \}^*)_{\phi,q_2} \right) .
                \\
  \end{align*}
Therefore, since the rearranged sets are totally ordered,          
    \begin{align*} 
    \mu (\{ Q_\phi f < \lambda \} )
            & \geq 
                \mu \left( \bigcup_{q \in S(\lambda)} (\{ f < q_1 \}^*)_{\phi,q_2} \right)
                \\
            &=
                \mu \left( \bigcup_{q \in S(\lambda)} (\{ f_* < q_1 \})_{\phi,q_2} \right)
                    \\
            &=
                \mu (\{ Q_\phi f_* < \lambda \} ).
    \end{align*}
    where the first equality is \eqref{eq: increasing rearrangement characterization}, and the last equality is a second application of Lemma \ref{lem: decomposition}.

    The second part of the theorem follows the same line and uses the same arguments
       \begin{align*}
        \mu (\{ Q_\phi f \geq  \lambda \}) 
            &=
        \mu (\{ Q_\phi f <  \lambda \}^c)    
          =       \mu \left( \bigcap_{q \in S(\lambda)} \{ f < q_1 \}_{\phi,q_2}^c \right)
                    \\
            &\leq 
                \inf_{q \in S(\lambda)} \mu  \left( \{ f < q_1 \}_{\phi,q_2}^c \right)
            \leq 
                \inf_{q \in S(\lambda)} \mu  \left( (\{ f < q_1 \}^\star)_{\phi,q_2}^c \right)
                \\
            &=
                \mu \left( \bigcap_{q \in S(\lambda)} (\{ f < q_1 \}^*)_{\phi,q_2}^c \right)
            =
                \mu \left( \bigcap_{q \in S(\lambda)} (\{ f_* < q_1 \})_{\phi,q_2}^c \right)
                    \\
            &=
                \mu (\{ Q_\phi f_* < \lambda \}^c) 
                 =
                 \mu (\{ Q_\phi f_* \geq \lambda \}) 
                 .
    \end{align*}
    This completes the proof of the theorem.
\end{proof}

As an illustrative example, one may consider the spherically symmetric rearrangement. It is an isoperimetric rearrangement for the cost 
$\phi(x,y)=tG(|x-y|/t)$, for $t>0$ and $G$ convex increasing (see the previous section). This special choice, and in particular the presence of the apparently superfluous parameter $t>0$, will prove useful for applications in Section \ref{sec:HJ}. 
By Theorem \ref{th:isoperimetric-rearrangeement} we can conclude that  the infimum convolution operator 
$$
Q_tf(x) \coloneqq \inf_{y \in \mathbb{R}^n} \left\{ f(y) + tG \left(\frac{|x-y|}{t} \right)  \right\}
$$
satisfies
\begin{equation} \label{eq:inf-conv}
|\{ Q_t f < \lambda \}| \geq |\{ Q_t f_* < \lambda \}| 
\end{equation}
for any $\lambda \in \mathbb{R}$.


\medskip

We end this section with a final comment when an exact isoperimetric inequality need not hold. 
Besides the 3 classical examples given in the Section \ref{sec:notations-definitions}, there are some situations where the isoperimetric sets are known. To give a complete list of publications is out of reach. We refer the interested reader to \cite{ABFC} and references therein for recent results related to weighted measures on the upper half plane. For most of the exact isoperimetric inequalities the isoperimetric sets are totally ordered and the result of this paper apply. We will not continue in this direction and may instead consider situations where exact isoperimetric sets are not known. 
One important example is the class of strictly log-concave probability measures on $\mathbb{R}^{n}$
admitting the following density 
$$
\frac{d\mu}{dx}=e^{-V(x)}, 
\qquad x \in \mathbb{R}^n
$$ 
with $V''\geq \rho Id$, $\rho >0$, where $V''$ is a shorthand notation for the Hessian matrix of $V$, and where the inequality is understood in the quadratic sense.
Recall that 
$\Phi(t):=\int_{-\infty}^{t}\frac{1}{\sqrt{2\pi}}e^{-\frac{x^{2}}{2}}dx$, and set $\varphi=\Phi'$ for the density of the standard one dimensional Gaussian measure.
It is  well known that $\mu$ satisfies the following isoperimetric inequality 
\begin{equation} \label{eq:levy}
    \mu^+(A)
    \geq 
    c
    \varphi\circ\Phi^{-1}(\mu(A)), \qquad A\in\mathcal{B}(\mathbf{R}^{n})
\end{equation} 
for some constant $c=c(\rho) \in (0,\infty)$
\cite{bakry-ledoux} (see also \cite{milman2010,ivanisvili-volberg,duggal2024curvature}).
When $V(x)=|x|^2/2-\frac{n}{2}\log(2\pi)$ the measure
$\mu$ is nothing but the $n$ dimensional standard Gaussian and \eqref{eq:levy} holds with constant $c=1$.
This is the classical Borell-Sudakov-Tsirelson Gaussian isoperimetry. 
For other  (than Gaussian) strictly log-concave probability measures, 
\eqref{eq:levy} is not optimal \cite{bobkov-houdre}.

A second situation that was considered and studied in the literature pertains to measures possessing a product structure. More precisely, these measures are of the form 
$$
\frac{d\mu_{n}}{dx}
=
\prod_{i= 1}^{n}\frac{1}{Z_{\Psi}}e^{-\Psi(|x_{i}|)}
$$
where $\Psi:\mathbb{R}_{+}\rightarrow\mathbb{R}_+$ is an increasing convex function where $\Psi(0)=0$ and $\sqrt{\Psi}$ is concave (here $Z_\Psi$ is a normalization constant that turn $\mu_n$ into a probability measure). We refer the reader to \cite{roberto} for as account on isoperimetry for products of probability measures. The assumption on the potential $\Psi$ indicates that the measure is "between" exponential and Gaussian. In particular, it is not strictly convex and therefore this class of measures differ from the previous ones. 
Let $\rho=e^{-\Psi}/Z_\Psi$ and $H(t)=\int_{-\infty}^t \rho(x)dx$.
It has been shown \cite{bobkov-houdre97,BCR06,BCR07}
that there exists $c>0$ (independent of $n$) such that 
\begin{equation} \label{eq:iso-product}
\mu_{n}^+(A)\geq c \rho \circ H^{-1}(\mu_{n}(A)), \qquad  A\in\mathcal{B}(R^{n}) .
\end{equation}
Moreover, such an isoperimetric inequality for $c=1$ and $n \geq 2$ implies that $\mu$ is Gaussian \cite{bobkov-houdre}. Again, such an isoperimetric inequality is not exact and extremal isoperimetric sets are not known even for simple cases like the double sided exponential measure \cite{bobkov-houdre97,BCR06}.

Although the isoperimetric inequalities in the two classes of probability measures introduced above are not optimal, our approach on the rearrangement of the inf convolution operator is flexible enough to be applied in this situation as we explain in the following paragraph.

Observe first that the isoperimetric inequalities  
\eqref{eq:levy} and \eqref{eq:iso-product} can be reformulated in an integrated form that is more useful for our purpose. In fact, as is classical (see \textit{e.g.}\ \cite{ledoux}) \eqref{eq:levy} and \eqref{eq:iso-product} imply respectively 
$$
\mu(A_{\varepsilon})\geq\mu((A^*)_{c\varepsilon}) 
\qquad 
\mbox{and}
\qquad 
\mu_n(A_{\varepsilon})\geq\mu_n((A^*)_{c\varepsilon})
$$
where $A_{\varepsilon}:=\{x\in\mathbb{R}^{n},\exists y\in\mathbb{R}^{n}, \|x-y\|<\varepsilon\}$ and 
$A^*:=\{x\in\mathbb{R}^{n}:x_1<a\}$ with $a$ chosen such that
$\mu(A)=\mu(A^*)$ and similarly for $\mu_n$.
Now following the exact same proof of Theorem \ref{th:isoperimetric-rearrangeement} (with the cost $\phi(x,y)=t\|x-y\|$, details are left to the reader), the inf convolution operator
$$
Q_{t}f(x):=\sup_{y\in\mathbb{R}^{n}}f(y)-t\|x-y\|
$$
satisfies
$$
\mu\left(\{Q_{t}f>\lambda\}\right)
\geq
\mu\left(\{Q_{\frac{t}{c}}f_*>\lambda\}\right), \qquad t >0 
$$
and similarly for $\mu_n$. Thanks to Remark \ref{rem:transfer}, similar results hold for more general costs of the form $\phi(x,y)=\alpha(\|x-y\|)$.

\section{Spherically symmetric rearrangement of Transforms} \label{sec:transforms}

In this section we will prove a rearrangement result for some transforms that include the well known Legendre and polar transforms.

In all of this section, we consider a spherically symmetric rearrangement. In particular, $f^*$ and $A^*$ denote the spherically symmetric rearrangements of the function $f$ and of the set $A$ respectively as introduced in Section \ref{sec:notations-definitions}.

For a Borel measurable $f : \mathbb{R}^n \to \mathbb{R}$ and $\rho:\mathbb{R}\rightarrow\mathbb{R}$ an increasing function with $\rho(0)=0$, define the $\mathcal{T}$-transform of $f$,  $\mathcal{T}f:\mathbb{R}^n \to \mathbb{R} \cup\{\infty\}$
    \begin{equation*}
        \mathcal{T}f(x) = \sup_{y \in \mathbb{R}^n} \bigg[ \rho(\langle x , y  \rangle) -  f(y) \bigg], \qquad x \in \mathbb{R}^n
    \end{equation*}
where $\langle x , y  \rangle$ denote the Euclidean scalar product in $\mathbb{R}^n$.

Observe that when $\rho(x)=x$ the $\mathcal{T}$-transform
is the usual Legendre transform. 
Such transforms were considered in the literature in relation with functional forms of the Blaschke-Santalo inequality \cite{ball86,fradelizi-meyer,lehec}.

Recall that $|\cdot|$ stands for the Lebesgue measure.  
One of our main results about rearrangement is the following.
 
\begin{theorem} \label{thm: majorization rearrangement transform}
        If $f: \mathbb{R}^n \to [0,\infty)$ is convex, even and $f(0) = 0$, then
        \[
            | \{ \mathcal{T}f \leq \lambda \} | \leq  | \{ \mathcal{T}f_* \leq \lambda \} |,
            \qquad \lambda \in \mathbb{R}.
        \]
\end{theorem}

The proof of the theorem will be based on 
a formulation with rearrangement of the celebrated Blaschke-Santalo inequality from convex geometry that is now recalled. 
For a Borel set $K$ of $\mathbb{R}^n$, define the support function of $K$ by $h_K(x) = \sup_{y \in K} \langle x, y \rangle$, and the polar set
    \[
        K^\circ \coloneqq \{ x: h_K(x) \leq 1 \}
    \]
    when $K$ is non-empty, and $K^\circ = \mathbb{R}^n$ otherwise.
The Blaschke-Santalo inequality asserts that, for a convex symmetric (\textit{i.e.}\ $K=-K$) non-empty set $K$ with finite volume,
        \[
            |K||K^\circ| \leq |B|^2 
        \]
where we recall that $B$ is the Euclidean unit ball of $\mathbb{R}^n$.
We refer the reader to \cite{fradelizi-23} for an historical presentation and references.

Note that since $K^*$ is a ball,
$K^*=\left( \frac{|K|}{|B|}\right)^\frac{1}{n}B$
and therfeore $(K^*)^\circ = \left( \frac{|B|}{|K|}\right)^\frac{1}{n}B$ (since in general $(\lambda K)^\circ = \lambda^{-1} K^\circ$). Hence
$|K^*||(K^*)^\circ|=|B|^2$ and by the 
Blaschke-Santalo inequality (and since $|K^*|=|K|$)
$$
|K||K^\circ| \leq |K^*||(K^*)^\circ| = |K||(K^*)^\circ| 
$$
from which we infer that the Blaschke-Santalo inequality can be reformulated equivalently as a rearrangement inequality: for a symmetric convex set $K \subseteq \mathbb{R}^n$ with finite volume,
\begin{align} \label{thm: Blaschke-Santalo rearrangement bis}
    |K^\circ| \leq | (K^*)^\circ|.
\end{align}

\begin{proof}[Proof of Theorem \ref{thm: majorization rearrangement transform}]
    Let us first note that since $f$ takes its minimum $0$ at $0$, 
        \[
            \mathcal{T}f(x) = \sup_{y \in \mathbb{R}^n} \big[ \rho(\langle x,y\rangle) - f(y) \big] \geq \rho(\langle x ,0 \rangle) - f(0) = 0.
        \]
    Thus for $\lambda < 0$, $\{ \mathcal{T} f \leq \lambda \}$ is empty and there is nothing to prove.
   Hence we assume that $\lambda \geq 0$. Then with the cost function $\phi(x,y) = -\rho(\langle x,y \rangle$),
    \begin{align*}
        \{ \mathcal{T}f \leq \lambda \} 
            &=
                \{ Q_\phi f \geq - \lambda \}
                    \\
            &=
                \{ Q_\phi f < - \lambda \}^c
                    \\
            &=
               \left( \bigcup_{q \in S(-\lambda)} \{ f < q_1 \}_{\phi,q_2} \right)^c
               & (\mbox{by Lemma } \ref{lem: decomposition})
                    \\
            &=
                \bigcap_{q \in S(-\lambda)} \{ f < q_1 \}_{\phi,q_2}^c
    \end{align*}
    where 
    $S(-\lambda) = \{q = (q_1,q_2) \in \mathbb{Q}^2 :  q_1 + q_2 < - \lambda \}$.  
    Note that since $\{ f < q_1 \}$ is empty unless $q_1 > 0$, it suffices to consider the union over $q \in S_+(-\lambda)$, where $S_+(-\lambda) = S(-\lambda) \bigcap \{(x_1,x_2): x_1 > 0 \}$. 
    Summarizing, we have
    \begin{align} \label{eq: exp transform sublevel set}
        \{ \mathcal{T} f \leq \lambda \} = \bigcap_{q \in S_+(-\lambda)} \{ f < q_1 \}_{\phi,q_2}^c .
    \end{align}
    Note that for $q \in S_+(-\lambda)$, since $q_1 \geq 0$ and $q_1 + q_2 < -\lambda$, $q_2$ is necessarily negative.
    We claim that for $\varepsilon >0$, and a Borel set $K$,
     \begin{align*} 
        (K_{\phi,-\varepsilon})^c = \rho^{-1}(\varepsilon) \ K^\circ
    \end{align*} 
    where we defined 
    $$
    \rho^{-1}(\varepsilon)=\inf \{y \in\mathbb{R}:\rho(y)\geq\varepsilon\} .
    $$
    Indeed,
    \begin{align*}
        (K_{\phi,-\varepsilon})^c
            &=
                \{ x : \exists y \in K, \phi(x,y) < - \varepsilon \}^c
                    \\
            &=
                \left\{ x: \inf_{y \in K} - \rho(\langle x,y \rangle) < - \varepsilon \right\}^c
                    \\
            &=
                \left\{ x : \sup_{y \in K} \rho(\langle x,y \rangle) \leq \varepsilon \right\}
                    \\
            &=
                \left\{ x : \sup_{y \in K} \langle x,y \rangle \leq \rho^{-1}(\varepsilon) \right\}
                    \\
            &=
                \rho^{-1}(\varepsilon) \left\{ x : \sup_{y \in K} \langle x,y \rangle \leq 1 \right\}
                    \\
            &=
                \rho^{-1}(\varepsilon) \ K^\circ.
    \end{align*}
    Applying the claim to \eqref{eq: exp transform sublevel set}, we have 
    \begin{align} \label{eq:sub-level-description}
        \{ \mathcal{T} f \leq \lambda \} = \bigcap_{q \in S_+(-\lambda)} \rho^{-1}(|q_2|) \{ f < q_1 \}^\circ .
    \end{align}
    Thus,
    \begin{align*}
        |\{ \mathcal{T} f \leq \lambda \}| 
            &=
                \left|  \bigcap_{q \in S_+(-\lambda)} \rho^{-1}(|q_2|) \{ f < q_1 \}^\circ \right|
                    \\
            &\leq 
                \inf_{q \in S_+(-\lambda)} \rho^{-1}(|q_2|)^n \ \big| \{f < q_1 \}^\circ \big|
                    \\
            &\leq 
                \inf_{q \in S_+(- \lambda)} \rho^{-1}(|q_2|)^n \ \big| (\{ f < q_1 \}^*)^\circ \big |
                & (\mbox{by Blaschke-Santalo's inequality } \eqref{thm: Blaschke-Santalo rearrangement bis})
                    \\
            &= 
                \inf_{q \in S_+(- \lambda)}  \big| \rho^{-1}(|q_2|) \ (\{ f < q_1 \}^*)^\circ \big |
                & (\mbox{by homogeneity of the measure})
                    \\
            &=
                \left|\bigcap_{q \in S_+(- \lambda)} \rho^{-1}(|q_2|)  \ (\{ f < q_1 \}^*)^\circ \right|
                    \\
            &=
                \left|\bigcap_{q \in S_+(- \lambda)} \rho^{-1}(|q_2|) \ (\{ f_* < q_1 \})^\circ \right|
                    \\
            &=
                | \{ \mathcal{T} f_* \leq \lambda \} |
    \end{align*}
    where the first inequality is $\mu( \cap_n A_n) \leq \inf_n \mu(A_n)$ combined with an application of the homogeneity of the Lebesgue measure.  
    The antepenultimate equality is due to the fact that we are considering an intersection of centered balls which are totally ordered. The penultimate equality is the characterization of the spherically symmetric increasing rearrangement and the final one is \eqref{eq:sub-level-description} applied to the spherically symmetric increasing rearrangement $f_*$.
\end{proof}
We will now provide two direct applications of the above theorem.  
\begin{definition} \label{def:legendre}
    For a Borel measurable $f : \mathbb{R}^n \to \mathbb{R}$ define the Legendre transform of $f$,  $\mathcal{L}f:\mathbb{R}^n \to \mathbb{R}$
    \[
        \mathcal{L}f(x) = \sup_{y \in \mathbb{R}^n} \bigg[ \langle x , y  \rangle -  f(y) \bigg].
    \]
\end{definition}
 
\begin{theorem} \label{thm: majorization rearrangement Legendre}
        If $f: \mathbb{R}^n \to [0,\infty)$ is convex even and $f(0) = 0$, then
        \[
            | \{ \mathcal{L}f \leq \lambda \} | \leq  | \{ \mathcal{L}f_* \leq \lambda \} | , \qquad \lambda \in \mathbb{R}.
        \]
\end{theorem}
\begin{proof}
    It suffices to apply the above theorem to $\rho(x)=x$.
\end{proof}

\begin{definition}\label{def:polar}
    For a Borel measurable $f:\mathbb{R}^{n}\rightarrow[0,\infty)$ that is 
    convex, and takes the value zero at the origin, define the Polar transform of f to be $$
    f^\circ(x)=\sup_{y \in \mathbb{R}^n}\frac{\langle x,y \rangle-1}{f(y)}
    $$
    when $f$ is not identically zero, and $0^\circ=0$.
\end{definition}
Here, following \cite{artstein-milman}\footnote{see  more precisely \cite[Section 2.1]{artstein-milman}.},  we convey that $\frac{0}{0}=0$ and $\frac{+}{0}=+\infty$. Observe that necessarily $f^\circ(0)=0$.
As is standard in the literature, we denote the set of non-negative 
even 
convex functions on $\mathbb{R}^{n}$ taking the value zero at the origin\footnote{often called geometric convex functions.} by $\mathrm{Cvx}_0(\mathbb{R}^n)$.
\begin{theorem} \label{thm: majorization rearrangement Polar}
        For all $f\in\mathrm{Cvx}_0(\mathbb{R}^{n})$ it holds
        \[
            | \{f^\circ \leq \lambda \} | \leq  | \{ (f_*)^\circ \leq \lambda \} | ,
            \qquad \qquad \lambda \in \mathbb{R}.
        \] 
\end{theorem}
\begin{proof}
    The proof is based on the following two observations and an immediate application of Theorem \ref{thm: majorization rearrangement Legendre}. 
    One may observe the following for $f\in\mathrm{Cvx}_0(\mathbb{R}^{n})$
    \begin{enumerate}
        \item $\forall \lambda>0, \{f^\circ>\lambda\}=\lambda\{\mathcal{L}f>\frac{1}{\lambda}\}$;
        \item $\forall t\geq 0, (tA)^{c}=tA^{c}$.
    \end{enumerate} The argument is then as follows (using the homogeneity of the Lebesgue measure at the forth and fifth  line)
    \begin{equation*}
        \begin{split}
            \left|\{(f_{*})^\circ\leq\lambda\}\right|&=\left|\{(f_{*})^{\circ}>\lambda\}^{c}\right|\\
            &\stackrel{1.}{=}\left|\left(\lambda\{Lf_{*}>\frac{1}{\lambda}\}\right)^{c}\right|\\
            &\stackrel{2.}{=}\left|\lambda\{Lf_{*}>\frac{1}{\lambda}\}^{c}\right|\\
            &=\lambda^{n}\left|\{Lf_{*}\leq\frac{1}{\lambda}\}\right|\\
            &\stackrel{Th.\ \ref{thm: majorization rearrangement Legendre}}{\geq}\left|\lambda\{Lf\leq\frac{1}{\lambda}\}\right|\\
            &\stackrel{1.\& 2.}{=}\left|\{f^\circ\leq\lambda\}\right| .
        \end{split}
    \end{equation*}
\end{proof}

\section{Applications} \label{sec:applications}

In this section we collect some applications of our rearrangement results.
The first is about comparison of the solutions of a Hamilton-Jacobi equation with initial data $f$ versus $f_*$. 
In the second subsection, we provide a simple proof of a result by Ben Li \cite{li} about extremizers in functional Blaschke-Santalo type inequalities. Finally, we reduce the very recent semi-group proof by Cordero et al.
\cite{cordero} on functional Blaschke-Santalo by heat-flow, to dimension 1, therefore, simplifying their final argument 
from the use of the variance Brascamp-Lieb inequality to the trivial observation that a square is non-negative.

\subsection{Hamilton-Jacobi equation and the Hopf-Lax formula}\label{sec:HJ}

In this section, we prove a comparison result for the solution of the following Hamilton-Jacobi equation
\begin{equation} \label{eq:hj}
\begin{cases}
u_t(x) + A(\nabla u_t) = 0 & \mbox{in } \mathbf{R}^n \times (0 ,\infty) \\
u_0=f & \mbox{on } \mathbf{R}^n \times \{t =0\} 
\end{cases}
\end{equation}
with $A \colon \mathbb{R}^n \to \mathbb{R}$ a potential defined as $A(x)=H(\|x\|)$, with $H$ convex, smooth and satisfying $\lim_{x \to \infty} H(x)/x=+\infty$. This is a first order parabolic equation with initial data (at $t=0$) $f$.
Above $u(t,x)$ (for short $u_t(x)$) is the unknown,  function of the two variables $t \in [0,\infty)$ and $x \in \mathbb{R}^n$ and the gradient is acting on the space variable. Here, $\|\cdot\|$ denotes the usual Euclidean norm.

Such PDEs are well understood and we refer to the book by Evans \cite{evans} for an introduction, a historical presentation, and references and to the book by Villani \cite{villani} for connections with the transport of mass phenomenon.
In particular, if $f$ is assumed to be Lipschitz continuous (\textit{i.e.}\ such that $\sup_{\genfrac{}{}{0pt}{}{x,y\in \mathbb{R}^n}{x\neq y}} \frac{|f(x)-f(y)|}{\|x-y\|} < \infty$), the following function 
$$
Q_tf(x) \coloneqq \inf_{y \in \mathbb{R}^n} \left\{ f(y) + t\mathcal{L}H \left(\frac{|x-y|}{t} \right)  \right\}, \qquad t \geq 0
$$
solves the Hamilton-Jacobi equation in the sense that
$Q_tf$ is Lipschitz continuous, differentiable almost everywhere in $\mathbb{R}^n \times(0,\infty)$, it satisfies $u_t(x) + A(\nabla u_t) = 0$ almost everywhere in $\mathbf{R}^n \times (0 ,\infty)$
and $Q_0f=f$ (such a solution is called a generalized solution). Furthermore, it is the unique solution of \eqref{eq:hj} in the viscosity sense. Above $\mathcal{L}H$ is the Legendre transform of $H$ in the sense of Definition \ref{def:legendre}.
The formula that defines $Q_tf$ is known as the Hopf-Lax formula.

Our aim is to compare the solution of \eqref{eq:hj} to the solution of the rearranged problem
\begin{equation} \label{eq:hj-rearranged}
\begin{cases}
v_t(x) + A(\nabla v_t) = 0 & \mbox{in } \mathbf{R}^n \times (0 ,\infty) \\
v_0=f_* & \mbox{on } \mathbf{R}^n \times \{t =0\} 
\end{cases}
\end{equation}
where $f_*$ is the spherically symmetric increasing rearrangement of $f$.

The first occurrence of a comparison problem for Hamilton Jacobi equations apparently appeared in a paper by Ferrone, Posteraro and Volpicelli \cite{FPV} where the potential $H$ is essentially $H(x)=\|x\|$ (see also \cite{giarrusso-nunziante,giarrusso,mercaldo} for related result), followed by 
Alvino, Ferrone, Trombetti and Lions \cite{AFTL} who extended the comparison results of \cite{FPV} to a larger class of $H$ that are assumed to be $1$ homogeneous ($H(tx)=|t|H(x), t \in \mathbb{R}$, $x \in \mathbb{R}^n$) and satisfying $a\|x\| \leq H(x) \leq b\|x\|$ for some constants $a,b >0$. To be really precise, these authors deal with 
Problem \eqref{eq:hj} with a term $A(t,\nabla u_t)$ that satisfies   $A(t,x) \geq H(x)$ with $H(x)=\|x\|$ in \cite{FPV} and $H$ $1$-homogeneous and satisfying $a\|x\| \leq H(x) \leq b\|x\|$ in \cite{AFTL}. More importantly the initial data for the rearranged problem considered in this series of papers is the symmetric decreasing rearrangement $f^*$ rather than $f_*$, and the equation is considered on a general domain $\Omega \subset \mathbb{R}^n$.  Their technique relies on a careful use of the co-area formula together with the classical Euclidean isoperimetric inequality (or some of its variant through more general Brunn-Minkowski's inequality \cite{busemann}, see \cite[Sections 2 and 4.2]{AFTL}).

Here instead, we will exploit the fact the the solution of \eqref{eq:hj} is given by the explicit Hopf-Lax formula to prove a comparison theorem.
In fact, we don't need $H$ to be $1$-homogeneous but rather that it is convex and super linear at infinity
(to guarantee that $Q_tf$ is a generalized solution of Problem \eqref{eq:hj}). Therefore, our results go in a different direction than the above cited references and deal with a different class of potentials $H$.

Before moving to our comparison result for the first order  Hamilton-Jacobi Problem \eqref{eq:hj}, let us briefly mention that the literature is huge on elliptic operators. After the pioneering work of Talenti \cite{talenti}, there have been a continuous and extended production of papers dealing with rearrangement and comparison theorems for solution of elliptic operators (the Laplacian being the model example). Talenti's argument relies, again, as main arguments, on the co-area formula together with the classical Isoperimetric inequality in the Euclidean space. One key observation by Talenti is also that, starting from a rearranged function leads to an equality in the successive inequalities used in the proof of the comparison result, therefore leading to explicit constants in various functional inequalities, with possibly a full description of the extremizers. It is impossible to cite and describe all papers that built upon Talenti's approach as the reader can convince him/herself by consulting the list of more than 200 references published before 2000 contained in \cite{trombetti}\footnote{It is also very hard to cite just a few.}. Let us just mention that people were interested in comparison results for more general elliptic PDE, linear and non linear, degenerate,etc.

\medskip

To make use of our general argument, we consider 
$\mathcal{X}=\mathcal{Y}=\mathbb{R}^n$ equipped with the Lebesgue measure. The Hopf-Lax formula above corresponds to the infimum convolution
$Q_\phi$ given in Definition \ref{def:inf-convolution}
with $\phi (x,y)= t  \mathcal{L}H \left( \frac{\|x-y\|}{t} \right)$, $x,y \in \mathbb{R}^n$, $t>0$ being a fixed parameter. Rearrangement are understood in the usual sense (symmetric decreasing with respect to the Lebesgue measure). Then the usual isoperimetric inequality guarantees that the triple $*$, the Lebesgue measure and $\phi$ is an isoperimetric rearrangement in the sense of Definition  \ref{def:isoperimetric-rearrangement}.

One of our main theorem is the following.

\begin{theorem} \label{thm:HJ}
Let $H \colon \mathbb{R}_+ \to \mathbb{R}$ be smooth, convex and super-linear ($H(x)/x \to \infty$ when $x \to \infty$). Let $f$ be a Lipschitz continuous function on $\mathbb{R}^n$ and $f_*$ be its symmetric increasing rearrangement.

Denote by $u_t$ and $v_t$ the solution of Problem 
\eqref{eq:hj} and \eqref{eq:hj-rearranged} respectively. Then, for all $\lambda \in \mathbb{R}$, it holds
$$
|\{ v_t < \lambda \}| \leq |\{u_t < \lambda\}| .
$$
\end{theorem}



\begin{proof}
From Theorem \ref{th:isoperimetric-rearrangeement} (see \eqref{eq:inf-conv}) we know that
$$
|\{ Q_\phi f_* < \lambda \}| \leq |\{Q_\phi f < \lambda\}| .
$$
The expected result follows at once since by construction $u_t=Q_\phi f$ and $v_t=Q_\phi f_*$ thanks to Lemma \ref{lem:Lipschitz} below (that guarantees that $f_*$ is Lipschitz continuous and therefore that $Q_\phi f_*$ is indeed solution of the rearranged Hamilton-Jacobi equation.
\end{proof}

The proof of Theorem \ref{thm:HJ} relied on the fact that Lipschitz continuity is preserved by rearrangement, a property that is of independent interest and that is most probaly well-known. We give a proof for completeness. Recall that a function on $\mathbb{R}^n$ is $L$-Lipschitz (or Lipschitz continuous with Lipschitz constant $L$) if $\sup_{x \neq y}\frac{|f(x)-f(y)|}{\|x-y\|} \leq L$.

\begin{lemma} \label{lem:Lipschitz}
Let $f \colon \mathbb{R}^n \to \mathbb{R}$ be $L$-Lipschitz with $L \in (0,\infty)$. Then its symmetric increasing rearrangement $f_*$ is also $L$-Lipschitz.
\end{lemma}

\begin{proof}
   We follow some ideas from \cite{MDA25}.
   We claim that a function $f$ is $L$-Lipschitz if and only if for any $\lambda \in \mathbb{R}$, any $\varepsilon>0$, 
   $$
   \{f < \lambda\}_\varepsilon \subset \{f < \lambda +L\varepsilon\},
   $$
   where we recall that $A_\varepsilon=A+\varepsilon B$ with $B$ the Euclidean unit ball. Strict inequality is needed. Indeed assume that $f$ is $L$-Lipschitz. Then if $x \in \{f < \lambda\}_\varepsilon$, by definition, there exists
   $y$ with $f(y)<\lambda$ such that $\|x-y\| < \varepsilon$. In particular,
   $$
   f(x) 
   = 
   f(y) + f(x) - f(y)  
   <
   \lambda + 
   L\|x-y\| 
   <
   \lambda + L \varepsilon  
   $$
which guarantees that $x \in \{f < \lambda +L\varepsilon\}$.

Assume reciprocally that $\{f < \lambda\}_\varepsilon \subset \{f < \lambda +L\varepsilon\}$ and take $x \neq y$. Fix $\delta,\eta>0$. By definition of the enlargement, $x \in \{ f < f(y)+\eta\}_{\|x-y\|+\delta}$ (because $f(y)<f(y)+\eta$). Therefore 
$f(x) < f(y)+\eta + L(\|x-y\|+\delta)$. By symmetry in $x,y$, we conclude that, for any $\delta,\eta>0$ it holds 
$|f(x)-f(y)| < \eta + L(\|x-y\|+\delta)$ from which we get that $f$ is $L$-Lipschitz and therefore the claim.

Now the key observation is given in \eqref{eq: increasing rearrangement characterization} as $\{f_* < \lambda\}=\{f<\lambda\}^*$. Fix $\lambda \in \mathbb{R}$, we will prove that 
$\{f_* < \lambda\}_\varepsilon \subset \{f_* < \lambda +L\varepsilon\}=\{f < \lambda +L\varepsilon\}^*$.
Observe that the set 
$$
\{f_* < \lambda\}_\varepsilon =
\{f_* < \lambda\} + \varepsilon B
=
\{f < \lambda\}^* + \varepsilon B
$$ 
is a centered ball. Since  $\{f < \lambda +L\varepsilon\}^*$ is also a centered ball, the desired inclusion 
$\{f_* < \lambda\}_\varepsilon \subset \{f < \lambda +L\varepsilon\}^*$ will follow by comparison of the volume of these two balls. Now, by the equality case in the Brunn-Minkowski's inequality and then Brunn-Minkowski's inequality (recall that $|\cdot|$ denotes the volume with respect to the Lebesgue measure), it holds
\begin{align*}
|\{ f_* < \lambda\}_\varepsilon \}|^\frac{1}{n} 
& = |\{ f < \lambda\}^* + \varepsilon B |^\frac{1}{n} \\\
& = 
|\{ f < \lambda\}^*|^\frac{1}{n} + \varepsilon |B|^\frac{1}{n} & \mbox{(equality case in Brunn-Minkoswki)} \\
& 
=
|\{ f < \lambda\}|^\frac{1}{n} + \varepsilon |B|^\frac{1}{n}  & \mbox{(by construction of the rearrangement)} \\
& \leq 
|\{ f < \lambda\}+ \varepsilon B|^\frac{1}{n} & \mbox{(Brunn-Minkoswki)} \\
&= 
|\{ f < \lambda\}_\varepsilon|^\frac{1}{n} \\
& \leq 
|\{ f < \lambda+L\varepsilon\}|^\frac{1}{n}
& \mbox{(Since } f \mbox{ is } L-\mbox{Lipschitz)} \\
& =
|\{ f < \lambda+L\varepsilon\}^*|^\frac{1}{n} & \mbox{(by construction of the rearrangement)} \\
& =
|\{ f_* < \lambda+L\varepsilon\}|^\frac{1}{n}
& \mbox{(by \eqref{eq: increasing rearrangement characterization})}.
\end{align*}
The volume are therefore comparable and hence $\{f_* < \lambda\}_\varepsilon \subset \{f_* < \lambda +L\varepsilon\}$ as announced, leading to the desired conclusion of the lemma. 
\end{proof}

\subsection{Extremal Functions in the Functional Santalo Inequality}

In all the following we will use $\mathcal{A}$ to denote either the Legendre transform ($\mathcal{L}$, Definition \ref{def:legendre}) or the Polar transform ($^\circ$, Definition \ref{def:polar}). Our aim is to recover, in a shorter way, the following result of Ben Li, using rearrangement. Recall that we denote $\mathrm{Cvx}_0(\mathbb{R}^{n})$ the set of non-negative even convex functions on $\mathbb{R}^{n}$ taking the value zero at at the origin. 
Let $\omega_{n}=\pi^{n/2}/\Gamma(\frac{n}{2}+1)$ denote the volume of the $n$ dimensional Euclidean unit ball.

\begin{theorem}[\cite{li}] \label{thm: existence of extremizer}
    There exists a radially symmetric extremizer of the following variational problem 
    $$
    \sup_{f\in\mathrm{Cvx}_0(\mathbb{R}^{n})}\int e^{-f}\int e^{-\mathcal{A}f}
    $$ 
    where the integrals are implicitly assumed to be taken with respect to the Lebesgue measure on $\mathbb{R}^{n}$. 
In particular
$$
    \sup_{f\in \mathrm{Cvx}_0(\mathbb{R}^{n})} 
    \int e^{-f}\int e^{-\mathcal{L}f}
    =
    (n\omega_{n})^2 
    \sup_{\Psi} 
    \int_0^\infty e^{-\Psi(r)}r^{n-1}dr 
    \int_0^\infty e^{-\mathcal{L}\Psi}r^{n-1}dr 
$$
and
$$
    \sup_{f\in \mathrm{Cvx}_0(\mathbb{R}^{n})} 
    \int e^{-f}\int e^{-f^\circ}
    =
    (n\omega_{n})^2 
    \sup_{\Psi} 
    \int_0^\infty e^{-\Psi(r)}r^{n-1}dr 
    \int_0^\infty e^{-\Psi^\circ(r)}r^{n-1}dr
$$
where the supremum runs over all convex non-decreasing function $\Psi \colon [0,\infty) \to [0,\infty)$ with $\Psi(0)=0$, and 
$\mathcal{L}\Psi(r)=\sup_{s>0}\{rs-\Psi(s)\}$, 
$\Psi^\circ(r) \coloneqq \sup_{s>0} \frac{rs-1}{\Psi(s)}$. 
\end{theorem} 

The theorem is a direct consequence of a series of elementary lemmas.

Let $\mathcal{M}$ be the set of absolutely continuous measures on $[0,\infty)$ (absolutely continuous with respect to the Lebesgue measure) whose density is log-concave. Furthermore, we define the class $\mathcal{S}$ in the following manner 
$$
\mathcal{S} :=\{\mu \in \mathcal{M} :  \frac{d\mu}{dx}=e^{-\Psi(x)}
\mbox{ with } \Psi(0)=0, \Psi \mbox{ is non-decreasing  and } \int_{0}^{\infty}\!e^{-\Psi(r)}r^{n-1}dr=1\} .
$$
Note that log-concave just says that $\Psi \colon [0,\infty) \to \mathbb{R}$ is convex. We need some preparation. 
Since $\Psi$ is non-decreasing, it holds  
    $$
    1
    =
    \int_{0}^{\infty}e^{-\Psi(r)}r^{n-1}dr\geq\int_{0}^{2}e^{-\Psi(r)}r^{n-1}dr
    \geq 
    e^{-\Psi(2)}\frac{2^{n}}{n} .
    $$ 
    In particular $\Psi(2) \geq \log\left(\frac{2^{n}}{n}\right)$. Set 
    \begin{equation} \label{eq:def_a}
      a \coloneqq 
      \frac{1}{2}\log\left(\frac{2^{n}}{n}\right)    
    \end{equation}
    and notice that $a>0$ does not depend on $\Psi$. Now, observe that since $\Psi$ is convex and $\Psi(0)=0$, the map $(0,\infty) \ni r \mapsto \Psi(r)/r$ is non-decreasing so that, for any $r \geq 2$, (again, uniformly in $\Psi$)
    \begin{equation} \label{eq:psi}
    \Psi(r) \geq \frac{\Psi(2)}{2}r \geq ar .
    \end{equation}

\begin{lemma} \label{lem:tight}
    The set $\mathcal{S}$ is tight.
\end{lemma}
\begin{proof}
    To prove tightness we need to show that, for all 
    $\varepsilon>0$, there exists a compact set $K_{\varepsilon} \subset [0,\infty)$ such that $\mu([0,\infty)\setminus K_{\varepsilon})<\varepsilon$ for all $\mu\in\mathcal{S}$.  
    
    Fix $\varepsilon >0$ and take $K_{\varepsilon}=[0,b]$ with $b \geq 2$ chosen so that $\frac{e^{-ab}}{a} < \varepsilon$ (and $a$ defined in \eqref{eq:def_a}). Then, 
    it holds
    $$
    \mu([0,\infty)\setminus K_{\varepsilon})
    =
    \int_{b}^{\infty}e^{-\Psi(r)}dr
    \leq
    \int_{b}^{\infty}e^{-ar}dr
    =
    \frac{e^{-ab}}{a} 
    <
    \epsilon 
    $$
    as expected.
\end{proof} 

\begin{lemma}\label{lem:absolutely-continuous}
If $\mu_{n}\in\mathcal{S}$ converges weakly to a measure $\mu$,  then $\mu$ is absolutely continuous with respect to the Lebesgue measure. Moreover, if $e^{-\Psi_n}$ denotes the density of $\mu_n$ and $e^{-\Psi}$ the density of $\mu$, then $\Psi_n$ converges point-wise to $\Psi$ on the positive axis.
\end{lemma}

\begin{proof}
Denote by $e^{-\Psi_n}$ the density of $\mu_n$ with respect to the Lebesgue measure,
on $[0,\infty)$. Since $\mu_n \in \mathcal{S}$ we have $\mu_n(U) = \int_U e^{-\Psi_n(x)} dx \leq |U| e^{- \Psi_n(0)} = |U|$ for any Borel set $U$.  Thus for an open set $U$ such that $A \subseteq U$, by the weak convergence,
\begin{align*}
    \mu(A) \leq \mu(U) \leq \liminf_n \mu_n(U) \leq \liminf_n |U| = |U|.
\end{align*}
By the regularity of the Lebesgue measure, taking the infimum over all such $U$ gives,
\[
    \mu(A) \leq |A|.
\]
Thus $\mu$ has a density with respect to the Lebesgue measure, as expected. 

We will now show that $\lim_{n\rightarrow\infty}\Psi_{n}(x)=\Psi(x)$ for all $x \in [0,\infty)$.
By Helly's selection theorem we can extract from the sequence $(e^{-\Psi_n})$ a subsequence $(e^{-\Psi_{\varphi(n)}})$ converging point-wise to $e^{-\Psi}$.
In particular $\Psi$ is convex and therefore continuous on $(0,\infty)$.  

Remark that for all $a,b>0$, weak convergence implies that $\lim_{n\rightarrow\infty}\mu_{n}(a,b)=\mu(a,b)$. By contradiction, assume that $\limsup_{n\rightarrow\infty}e^{-\Psi_{n}(r)}>e^{-\Psi(r)}$ for some $r > 0$. 
The strict inequality guarantees  the existence of $\epsilon>0$ such that $\limsup_{n\rightarrow\infty}e^{-\Psi_{n}(r)} > \varepsilon + e^{-\Psi(r)}$. By monotonicity of $\Psi_n$, for any $0<a<r$ it holds
     \begin{align*}
      (r-a)(\varepsilon + e^{-\Psi(r)})
      < (r-a) \limsup_{n\rightarrow\infty} e^{-\Psi_n(r)} 
      \leq 
      \limsup_{n\rightarrow\infty}\mu_{n}(a,r)
      =\mu(a,r) 
      \leq (r-a)e^{-\Psi(a)} .
      \end{align*} 
      Therefore $\varepsilon + e^{-\Psi(r)} < e^{-\Psi(a)}$ that leads, by continuity of $\Psi$, in the limit $a \to r$, to a contradiction. Hence,  for all $r >0$, $\limsup_{n\rightarrow\infty}e^{-\Psi_{n}(r)} \leq e^{-\Psi(r)}$.
     
     Assume now that  $\liminf_{n\rightarrow\infty}e^{-\Psi_{n}(r)}<e^{-\Psi(r)}$ for some $r > 0$. Similarly, for some $\varepsilon>0$ and any $a>r$ it holds
  \begin{align*}
      (a-r)(-\varepsilon + e^{-\Psi(r)})
      > (a-r) \liminf_{n\rightarrow\infty} e^{-\Psi_n(r)} 
      \geq 
      \limsup_{n\rightarrow\infty}\mu_{n}(r,a)
      =\mu(r,a) 
      \geq (a-r)e^{-\Psi(a)} .
      \end{align*}    
This leads in the limit $a \to r$ to a contradiction and therefore to the desired conclusion: $\Psi_n(x) \to \Psi$ point-wise on $(0,\infty)$ and hence on $[0,\infty)$ since $\Psi_n(0)=\Psi(0)=0$.

\end{proof}

\begin{lemma} \label{lem:compact}
    The set $\mathcal{S}$ is compact  for the topology of the weak convergence.
\end{lemma}

\begin{proof}
    Consider a sequence $(\mu_n)$ of $\mathcal{S}$. We will prove that we can extract a converging subsequence (for the weak convergence) to a limit $\mu \in \mathcal{S}$.
     Denote by $e^{-\Psi_n}$ the density of $\mu_n$ with respect to the Lebesgue measure,
on $[0,\infty)$. Since $\mu_n \in \mathcal{S}$, $e^{-\Psi_n} \in [0,1]$ is non-increasing. Applying Helly's selection theorem we can extract a subsequence
$e^{-\Psi_{\varphi(n)}}$ converging pointwise to some limit $e^{-\Psi}$. Now by Lemma \ref{lem:tight} and Prokhorov's theorem we can further extract from the sub-sequence $(\mu_{\varphi(n)})$ a sub-subsequence weakly converging to some measure $\mu$ (with density $e^{-\Psi}$ by Lemma \ref{lem:absolutely-continuous}). Let's call for simplicity again $\mu_n$ this converging sub-subsequence.

To prove that $\mathcal{S}$ is compact, it only remains to prove that the limit $\mu$ belongs to $\mathcal{S}$. 

By the point-wise convergence we infer that $\Psi(0)=0$ and that $\Psi$ is non-decreasing and convex. We therefore only need to check that
$$
\int_{0}^{\infty}e^{-\Psi(r)}r^{n-1}dr=1.
$$
Recall the definition of $a$ defined in \eqref{eq:def_a}
and set $g(r)=r^{n-1} \left(\mathds{1}_{[0,2]} + e^{-ar}\mathds{1}_{(2,\infty)}\right)$. From \eqref{eq:psi} we see that $\Psi_m \leq g$ for all $m$, and since $g$ is integrable on the positive axis, by the dominated convergence theorem, we finally obtain
$$
1 = \lim_{m \to \infty} \int_0^\infty e^{-\Psi_m(r)}r^{n-1}dr = \int_0^\infty e^{-\Psi(r)}r^{n-1}dr
$$
as expected. This ends the proof of the Lemma.
\end{proof}
    
We are now in position to proove Theorem \ref{thm: existence of extremizer}.

\begin{proof} [Proof of Theorem \ref{thm: existence of extremizer}]
In the variational problem $\sup_{f\in\mathrm{Cvx}_0(\mathbb{R}^{n})}\int e^{-f(x)}dx\int e^{-\mathcal{A}f(x)}dx$ under consideration, $e^{-f}$ and $e^{-\mathcal{A}f}$ are non-negative functions. Therefore, we use the layer cake decomposition to get
\begin{align*}
\int e^{-f(x)}dx\int e^{-\mathcal{A}f(x)}dx
& =
\int_0^\infty |\{e^{-f} > t\}|dt \int_0^\infty |\{e^{-\mathcal{A}f} > t\}|dt \\
& =
\int_0^\infty |\{f < -\log t\}|dt \int_0^\infty |\{\mathcal{A}f < -\log t\}|dt .
\end{align*}
Now by \eqref{eq: increasing rearrangement characterization} and the definition of the rearrangement, it holds
$$
|\{f < -\log t\}|=|\{f < -\log t\}^*|=|\{f_* < -\log t\}|
$$ 
for any $t>0$. Therefore, using Theorem \ref{thm: majorization rearrangement Legendre}/Theorem \ref{thm: majorization rearrangement Polar} (according to the transform $\mathcal{A}$ under consideration), we get
$$
\int e^{-f(x)}dx\int e^{-\mathcal{A}f(x)}dx
\leq 
\int_0^\infty |\{f_* < -\log t\}|dt \int_0^\infty |\{\mathcal{A}f_* < -\log t\}|dt .
$$
Using the layer cake decomposition backward it follows that
$$
\int e^{-f(x)}dx\int e^{-\mathcal{A}f(x)}dx
\leq 
\int e^{-f_*(x)}dx\int e^{-\mathcal{A}f_*(x)}dx .
$$
Hence, as a first intermediate conclusion, since $f_* \in \mathrm{Cvx}_0(\mathbb{R}^{n})$,
it holds
$$
\sup_{f \in \mathrm{Cvx}_0(\mathbb{R}^{n})} \int e^{-f(x)}dx\int e^{-\mathcal{A}f(x)}dx
=
\sup_{f \in \mathrm{Cvx}_0(\mathbb{R}^{n})} \int e^{-f_*(x)}dx\int e^{-\mathcal{A}f_*(x)}dx .
$$
Now we observe that for any $f \in \mathrm{Cvx}_0(\mathbb{R}^{n})$, there exists a convex non-decreasing function $\Psi \colon [0,\infty) \to [0,\infty)$, with $\Psi(0)=0$ such that $f_*(x)=\Psi(\|x\|)$, $x \in \mathbb{R}^d$. Moreover, it is not difficult to check that 
$\mathcal{A}f_*(x)=\mathcal{A}\Psi(\|x\|)$ for any $x \in \mathbb{R}^d$. Therefore, passing to polar coordinates, we end up with
$$
\sup_{f \in \mathrm{Cvx}_0(\mathbb{R}^{n})} \int e^{-f(x)}dx\int e^{-\mathcal{A}f(x)}dx
= (n\omega_n)^2
\sup_{\Psi} \int_0^\infty e^{-\Psi(r)}r^{n-1}dx
\int_0^\infty e^{-\mathcal{A}\Psi(r)}r^{n-1}dr 
$$
where the supremum is running over all $\Psi$ convex, non-decreasing and satisfying $\Psi(0)=0$, and where $\omega_n\coloneqq\pi^{n/2}/\Gamma(\frac{n}{1}+1)$ denotes the volume of the $n$ dimensional Euclidean unit ball. 

We also observe that the left and right hand side are invariant under scaling, \textit{i.e.}\ changing $f$ into $x \mapsto f(\lambda x)$, and similarly for $\Psi$ does not affect the products $\int e^{-f(x)}dx\int e^{-\mathcal{A}f(x)}dx$ and $\int_0^\infty e^{-\Psi(r)}r^{n-1}dx\int_0^\infty e^{-\mathcal{A}\Psi(r)}r^{n-1}dr$. Therefore, we can assume $\int e^{-\Psi(r)}r^{n-1}dr=1$ and 
$$
\sup_{f \in \mathrm{Cvx}_0(\mathbb{R}^{n})} \int e^{-f(x)}dx\int e^{-\mathcal{A}f(x)}dx
= (n\omega_n)^2 \sup_{\mu \in \mathcal{S}} \int_0^\infty e^{-\mathcal{A}\Psi(r)}r^{n-1}dr .
$$
To conclude the proof of the lemma we claim that the map $\mathcal{S} \ni \mu \mapsto \int_0^\infty e^{-\mathcal{A}\Psi(r)}r^{n-1}dr$  is upper semi continuous. Indeed, consider a weakly converging sequence $(\mu_m)$ of $\mathcal{S}$. By compactness of $\mathcal{S}$ (Lemma \ref{lem:compact}), $\mu \in \mathcal{S}$, and by Lemma \ref{lem:absolutely-continuous},  if $e^{-\Psi_m}$ denote the density of $\mu_m$ and $e^{-\Psi}$ the density of $\mu$, $\Psi_m$ converge point-wise to $\Psi$.
By Fatou's Lemma, it follows
$$
\limsup_{m\rightarrow\infty}\int_{0}^{\infty}e^{-\mathcal{A}\Psi_{m}(r)}r^{n-1}dr\leq\int_{0}^{\infty}\limsup_{m\rightarrow\infty}e^{-\mathcal{A}\Psi_{m}(r)}r^{n-1}dr
\leq 
\int_{0}^{\infty}e^{-\mathcal{A}\Psi(r)}r^{n-1}dr,
$$ 
where in the last inequality we used that $\liminf \mathcal{A}\Psi_m \geq \mathcal{A}\Psi$, a property that can easily be checked from the definition of the Legendre/polar transform, proving the claim. 
Since upper semi-continuous functions attains their maximum on compact sets, the proof of the theorem is complete.

\end{proof}

\subsection{Functional Blaschke-Santalo's inequality via Semi-group revisited}

In this section we will prove the following well-known functional Blaschke-Santalo's inequality \cite{ball86}: for all measurable super linear function $f \in \mathrm{Cvx}_0(\mathbb{R}^n)$ 
it holds
\begin{equation} \label{eq:functionalBS}
    \int e^{-f(x)}dx \int e^{- \mathcal{L}f(x)}dx \leq (2\pi)^n .
\end{equation}
Recall that $\mathrm{Cvx}_0(\mathbb{R}^{n})$ denotes the set of non-negative even convex functions on $\mathbb{R}^{n}$ taking the value zero at at the origin and a function is called super-linear if for all $x \in \mathbb{R}^n$, $f(x) \geq a|x|+b$ for some $a>0$, $b \in \mathbb{R}$. Super-linearity  will allow us to avoid some non-crucial  technicalities. We follow \cite{cordero} here, with this choice.

We stress that there is nothing  new in this section: our aim is to show how our rearrangement results developed earlier can simplify some aspect of known proofs. More specifically, following the semi-group argument of 
\cite{cordero}, thanks to our reduction to dimension 1, their final Brascamp-Lieb argument will reduce to the trivial fact that a square is non-negative. We may, at different step, refer to \cite{artstein-klartag-milman,cordero} for some technical details.

By Theorem \ref{thm: existence of extremizer}
we know that
$$
    \sup_{f\in \mathrm{Cvx}_0(\mathbb{R}^{n})} 
    \int e^{-f}\int e^{-\mathcal{L}f}
    =
    (n\omega_{n})^2 
    \sup_{\Psi} 
    \int_0^\infty e^{-\Psi(r)}r^{n-1}dr 
    \int_0^\infty e^{-\mathcal{L}\Psi(r)}r^{n-1}dr
$$
where the second supremum runs now over all super-linear\footnote{It is straight forward to adapt the proof of  Theorem \ref{thm: existence of extremizer} to take into account the extra super-linearity assumption considered here.}  convex non-decreasing function $\Psi \colon [0,\infty) \to [0,\infty)$ with $\Psi(0)=0$, and 
$\mathcal{L}\Psi(r)=\sup_{s>0}\{rs-\Psi(s)\}$. 
Super-linearity ensures in particular that the Laplace transform of $\Psi$ is finite on the whole $[0,\infty)$.

\subsubsection*{Technical preparation: the Bessel semi-group}

On $(0,\infty)$, define the following generator
$$
L^bf(r)=f''(r) +\frac{n-1}{r}f'(r), \qquad r>0
$$
where $f \colon (0,\infty) \to \mathbb{R}$ is smooth enough. We denote $(P_t^b)_{t \geq 0}$
its associated semi-group and observe that the reversible measure is $\nu(dr)=r^{n-1}dr$, $r \in (0,\infty)$. 
This is the projection of the heat semi-group $(P_t^h)_{t \geq 0}$ on radial functions, called the Bessel semi-group. More precisely, for $f \colon (0,\infty) \to \mathbb{R}$, define $\tilde{f}(x)=f(\|x\|)$, $x \in \mathbb{R}^n$ so that
$$
P_t^bf(\|x\|) = P_t^h \tilde f(x) = \tilde f* g_t(x), \qquad x \in \mathbb{R}^n ,
$$ 
with $g_t(x)\coloneqq e^{-\|x\|^2/(4t)}/(4\pi t)^{n/2}$.

Consider a super-linear, convex, non-decreasing function $\Psi \colon (0,\infty) \to \mathbb{R}$, with $\Psi(0)=0$, and set 
$$
\Psi_t(r)
\coloneqq
-\log P_t^b(e^{-\Psi})(r), \qquad r \geq 0 .
$$
for $\Psi$ defined on the positive axis, super-linear non-decreasing, with $\Psi(0) = 0$. 
Being a convolution with a Gaussian kernel $P_t^be^{-\Psi}$ is positive, $\mathcal{C}^\infty$ smooth. 
Moreover, super-linearity ensures that $\Psi_t$ is also super-linear (see \cite{cordero} for the details), and together with the convexity property, $\Psi_t$ is strictly convex \cite{cordero}. 
In particular, $\Psi_t'$ has an reciprocal function we denote ${\Psi_t'}^{-1}$ and, as is easy to check from the definition,  
$(\mathcal{L}\Psi_t)'={\Psi_t'}^{-1}$
and $\mathcal{L}\Psi_t(r) = r\Psi_t(r)- \Psi_t((\mathcal{L}\Psi_t)'(r)))$
on the positive axis. Taking the derivative in $t$, it follows that 
$$
\frac{d}{dt}\mathcal{L}\Psi_t(r) 
=
r \frac{d}{dt}(\mathcal{L}\Psi_t)'(r)
-\Psi_t'((\mathcal{L}\Psi_t)'(r)) \frac{d}{dt}(\mathcal{L}\Psi_t)'(r)
- \frac{d}{dt}\Psi_t ((\mathcal{L}\Psi_t)'(r)) 
=
- \frac{d}{dt}\Psi_t ((\mathcal{L}\Psi_t)'(r)) 
$$
where the last equality follows from the identity $\Psi_t'((\mathcal{L}\Psi_t)'(r)) =r$.
Since $\Psi_t'=- \frac{P_t^b(e^{-\Psi})'}{P_t^b(e^{-\Psi})}$ and
 $\Psi_t''=- \frac{P_t^b(e^{-\Psi})''}{P_t^b(e^{-\Psi})}+ \left(\frac{P_t^b(e^{-\Psi})'}{P_t^b(e^{-\Psi})} \right)^2$, it holds
 \begin{align*}
     \frac{d}{dt}\Psi_t (r)
     & = 
     -\frac{L^b P_t^b(e^{-\Psi})(r)}{P_t^b(e^{-\Psi})(r)} \\
     & =
     \frac{-1}{P_t^b(e^{-\Psi})(r)} 
     \left( P_t^b(e^{-\Psi})''(r) + \frac{n-1}{r} P_t^b(e^{-\Psi})'(r)\right) \\
     & =
     -\Psi_t'(r)^2 + \Psi_t''(r) + \frac{n-1}{r} \Psi_t'(r) .
 \end{align*}
 Therefore
 \begin{align*}
 \frac{d}{dt} \mathcal{L}\Psi_t(r) 
 & = - \frac{d}{dt}\Psi_t((\mathcal{L}\Psi_t)'(r)) 
  = 
 - \frac{d}{dt}\Psi_t({\Psi_t'}^{-1}(r)) \\
 & =
 r^2 - \Psi_t''({\Psi_t'}^{-1}(r)) - \frac{n-1}{{\Psi_t'}^{-1}(r)}r , \qquad r >0 .
 \end{align*}
 Now $\Psi_t'( ({\mathcal{L}\Psi_t)'}(r))=r$ guarantees that
 $\Psi_t''((\mathcal{L}\Psi_t)'(r))(\mathcal{L}\Psi_t)''(r)=1$, so that, since $(\mathcal{L}\Psi_t)'={\Psi_t'}^{-1}$,
 $\Psi_t''({\Psi_t'}^{-1}(r))=\frac{1}{(\mathcal{L}\Psi_t)''(r)}$.
 It follows that
 \begin{equation} \label{eq:cordero}
 \partial_t \mathcal{L}\Psi_t(r) 
 =
 -\frac{1}{(\mathcal{L}\Psi_t)''(r)} + r^2 - \frac{n-1}{(\mathcal{L}\Psi_t)'(r)}r , \qquad r >0 .
 \end{equation}

\subsubsection*{Proof of  the functional Blaschke-Santalo's inequality \eqref{eq:functionalBS}}

Following \cite{cordero}, our aim is to prove that
$
t \mapsto 
\int_0^\infty e^{-\Psi_t(r)}r^{n-1}dr  
\int_0^\infty e^{-\mathcal{L}\Psi_t(r)}r^{n-1}dr  
$
 is non-decreasing. 
This would imply that 
$$
\int_0^\infty \!e^{-\Psi(r)}r^{n-1}dr  
\int_0^\infty \!e^{-\mathcal{L}\Psi(r)}r^{n-1}dr
\leq 
\lim_{t \to \infty}
\int_0^\infty \!e^{-\Psi_t(r)}r^{n-1}dr  
\int_0^\infty \!e^{-\mathcal{L}\Psi_t(r)}r^{n-1}dr
= \frac{(2\pi)^n}{(n \omega_n)^2}
$$
 where the limit can easily be computed thanks to the previous technical preparation (see \cite{artstein-klartag-milman,cordero} for  details). 
 
Observe first that, by invariance,
\begin{equation} \label{eq:0}
\int_0^\infty e^{-\Psi_t(r)}r^{n-1}dr = 
\int_0^\infty P_t^b(e^{-\Psi})r^{n-1}dr 
=
\int_0^\infty e^{-\Psi(r)} r^{n-1}dr .
\end{equation}
Therefore the first factor in $\int_0^\infty e^{-\Psi_t(r)}r^{n-1}dr  
\int_0^\infty e^{-\mathcal{L}\Psi_t(r)}r^{n-1}dr$  does not depend on $t$ and we are reduced to proving that
$$
t \mapsto \alpha(t) \coloneqq \int_0^\infty e^{-\mathcal{L}\Psi_t(r)}r^{n-1}dr 
$$
is non-decreasing on the positive axis.
Now, thanks to \eqref{eq:cordero}, it holds
\begin{align*}
\alpha'(t) 
& = 
-   
    \int_0^\infty  \partial_t \mathcal{L}\Psi_t(r) e^{-\mathcal{L}\Psi_t(r)}r^{n-1}dr  \\
&    =
    \int_0^\infty   \left( \frac{1}{(\mathcal{L}\Psi_t)''(r)} - r^2 + \frac{n-1}{(\mathcal{L}\Psi_t)'(r)}r \right) e^{-\mathcal{L}\Psi_t(r)}r^{n-1}dr .     
\end{align*}
To ease notation, we set   $V=\mathcal{L}\Psi_t$ and observe that, $V',V''>0$. In order to handle the border case in the forthcoming integration by part, we introduce for $\varepsilon >0$,
$V_\varepsilon(r) \coloneqq V(r) + \varepsilon r$, $r \geq 0$. We claim that
\begin{equation} \label{eq:claim}
\int_0^\infty  \left( \frac{1}{V_\varepsilon''(r)} - r^2 + \frac{n-1}{V_\varepsilon'(r)}r \right) e^{-V_\varepsilon(r)}r^{n-1}dr     \geq 0 .
\end{equation}
This would imply $\alpha'(t) \geq 0$ by the monotone convergence theorem, in the limit $\varepsilon \downarrow 0$. Therefore to prove \eqref{eq:functionalBS} it only remains to establish \eqref{eq:claim}. To that aim, we use an integration by part to get 
\begin{align*}
-\int e^{-V_\varepsilon(r)}r^{n+1}dr
& =
\int -V_\varepsilon'(r)e^{-V_\varepsilon(r)} \frac{r^{n+1}}{V_\varepsilon'(r)}dr \\
& =
\left[ e^{-V_\varepsilon(r)}r^{n+1}/V_\varepsilon'(r)\right]_0^\infty 
- \int_0^\infty \left(
\frac{(n+1)r^n}{V_\varepsilon'(r)} - \frac{V''_\varepsilon(r)r^{n+1}}{V'_\varepsilon(r)^2} \right)
e^{-V_\varepsilon(r)} dr\\
& =
- \int_0^\infty 
\frac{(n-1)r}{V_\varepsilon'(r)} e^{-V_\varepsilon(r)} r^{n-1}dr
- \int_0^\infty \left(
\frac{2r}{V_\varepsilon'(r)} - \frac{V''_\varepsilon(r)r^2}{V_\varepsilon'(r)^2} \right)
e^{-V_\varepsilon(r)} r^{n-1}dr 
\end{align*}
where in the last equality we used that $V_\varepsilon$ is strictly convex and that $V_\varepsilon'(0) \geq \varepsilon$ so that the bracket term vanishes.
Therefore,
\begin{align*}
 & \int_0^\infty \left(\frac{1}{V_\varepsilon''(r)} -r^2 + \frac{n-1}{V_\varepsilon'(r)}r \right)  e^{-V_\varepsilon(r)}r^{n-1}dr \\
  & \qquad \qquad \qquad\qquad=
  \int_0^\infty \left(\frac{1}{V_\varepsilon''(r)} -\frac{2r}{V_\varepsilon'(r)} + \frac{V_\varepsilon''(r)r^2}{V_\varepsilon'(r)^2} \right) e^{-V_\varepsilon(r)}r^{n-1}dr \\
  &  \qquad \qquad \qquad\qquad=
  \int_0^\infty \frac{(rV_\varepsilon''(r)-V_\varepsilon'(r))^2}{V''_\varepsilon(r)V'_\varepsilon(r)^2} e^{-V_\varepsilon(r)}r^{n-1}dr \\
  & \qquad \qquad \qquad\qquad 
  \geq 0
\end{align*}
as expected.

\bibliographystyle{plain}
\bibliography{rearrangement}

\end{document}